\documentclass[12pt]{amsart}

\usepackage{amsfonts,amsmath,amssymb,amsthm,mathrsfs,url,color,latexsym,rawfonts,lscape,graphicx}
\usepackage{hyperref}
\usepackage{esint}

\theoremstyle{plain}
  \newtheorem{theorem}{Theorem}[section]
  \newtheorem{lemma}{Lemma}[section]

  \newtheorem{corollary}{Corollary}[section]
  
  \newtheorem{remark}{Remark}[section]

  \newtheorem*{thma}{Theorem 1.3\textprime}

   \newcommand{\beqn}{\begin{eqnarray}}
   \newcommand{\eeqn}{\end{eqnarray}}
   \newcommand{\beqs}{\begin{eqnarray*}}
   \newcommand{\eeqs}{\end{eqnarray*}}
   \newcommand{\ban}{\begin{eqnarray*}}
   \newcommand{\nan}{\end{eqnarray*}}
   \newcommand{\beq}{\begin{equation}}
   \newcommand{\eeq}{\end{equation}}

  \newcommand{\RR}{{\mathbb R}}
  \newcommand{\ZZ}{{\mathbb Z}}

  \newcommand{\R}{\RR}

\newcommand{\p}{\partial}

\newcommand{\Om}{\Omega}
\newcommand{\pom}{{\p\Om}}
\newcommand{\bom}{{\overline\Om}}

\renewcommand{\det}{\mbox{det}}

  \newcommand{\Vol}{{\mbox{Vol}}}

  \numberwithin{equation}{section}
  \numberwithin{figure}{section}

\begin{document}

\title[Boundary regularity for second BVP of Monge-Amp\`ere equations]
{Boundary regularity for the second boundary-value problem of Monge-Amp\`ere equations in dimension two}

%\date\today

\author[S. Chen]
{Shibing Chen}
\address
{Centre for Mathematics and Its Applications,
The Australian National University,
Canberra, ACT 0200, AUSTRALIA}
\email{Shibing.Chen@anu.edu.au}

\author[J. Liu]
{Jiakun Liu}
\address
	{%Institute for Mathematics and its Applications, 
	School of Mathematics and Applied Statistics,
	University of Wollongong,
	Wollongong, NSW 2522, AUSTRALIA}
\email{jiakunl@uow.edu.au}

\author[X.-J. Wang]
{Xu-Jia Wang}
\address
{Centre for Mathematics and Its Applications,
The Australian National University,
Canberra, ACT 0200, AUSTRALIA}
\email{Xu-Jia.Wang@anu.edu.au}

\thanks{This work was supported by ARC FL130100118 and ARC DP170100929.}



\begin{abstract}
In this paper, we introduce an iteration argument to prove that a convex solution to the Monge-Amp\`ere equation 
$\det D^2 u =f $ in dimension two subject to the natural boundary condition $Du(\Omega) = \Omega^*$ 
is $C^{2,\alpha}$ smooth up to the boundary. 
We establish the estimate under the sharp conditions that the inhomogeneous term $f\in C^\alpha$ 
and the domains are convex and $C^{1,\alpha}$ smooth. When $f\in C^0$ (resp.
$1/C<f<C$ for some positive constant $C$), we also obtain the global $W^{2,p}$ (resp. $W^{2,1+\epsilon}$) regularity. 
\end{abstract}

\maketitle

\baselineskip=16.4pt
\parskip=3pt

\section{Introduction}
In this paper we study the second boundary-value problem for the Monge-Amp\`ere equation, 
\beq\label{MAE1}
\left\{\begin{array}{ll}
\det D^2u = f &\quad\mbox{in }\Om,\\
Du(\Om) = \Om^*, &
\end{array}
\right.
\eeq
where $\Om, \Om^*$ are two bounded domains in $\R^2$ and $f$ is a given positive function. 
Problem \eqref{MAE1} received intensive attention in the last two decades, 
due to its wide range of applications, such as in optimal transportation, minimal Lagrangian graphs, 
seismology, image processing, machine learning, 
see for instances \cite{Bre, BW, EFY, HZTA, Vi} and references therein. 

The regularity of solutions to \eqref{MAE1} is a focus of attention in optimal transportation \cite {Vi}.
In a landmark paper \cite{C96}, Caffarelli established the global $C^{2,\alpha'}$ regularity for the problem \eqref{MAE1}, 
assuming that $\Om$ and $\Om^*$ are uniformly convex with $C^2$ boundary and $f\in C^\alpha(\overline\Om)$. 
Under stronger assumptions on domains and $f$, 
the global smooth solution was first obtained by Delano\"e \cite{De} in dimension two 
and later on by Urbas \cite{Ur} in higher dimensions. 
The uniform convexity and smoothness of domains played a critical role 
in the above mentioned papers.

Very recently in \cite{CLW} we discovered a different proof for the global $C^{2,\alpha}$ regularity, 
assuming the domains $\Om, \Om^*$ are convex only (instead of uniformly convex) with merely $C^{1,1}$ boundaries, 
and $f\in C^\alpha(\overline\Om)$ for the same $\alpha\in(0,1)$. 
Moreover, we obtained the global $W^{2,p}$ regularity for all $p\geq1$, when $f\in C^0(\overline\Om)$. 
We introduced a completely different method to obtain the uniform obliqueness, 
which is the key estimate in proving the boundary regularity. 

In particular in dimension two, it was proved that the uniform obliqueness holds under much weaker conditions,
including the case when $\Om$ and $\Om^*$ are bounded convex domains with $C^{1,\alpha}$ boundaries. 
See  \cite[\S4.1]{CLW}.
On the other hand, it was proved by Caffarelli \cite {C96} that
the uniform density  holds for general convex domains in dimensions two.
Combining the above properties, in this paper 
we are able to prove the following global regularity results for problem \eqref{MAE1} 
in dimension two.
The key point is that by using an iteration argument, we can improve the regularity from $u\in C^{1,\alpha}$ to $u\in C^{1,\alpha'}$ for a greater $\alpha'>\alpha$. By estimating the increment $|\alpha'-\alpha|\geq \delta$ for a universal constant $\delta>0$, we then obtain that $u\in C^{1,\alpha}$ for all $\alpha<1$.

\vspace{5pt}

The first one is the global $C^{2,\alpha}$ regularity under sharp conditions.

\begin{theorem}\label{thm1}
Assume that $\Omega, \Omega^*$ are two bounded convex domains in $\mathbb{R}^2$.
Assume that $\partial\Omega, \partial\Omega^* \in C^{1,\alpha}$, and $f\in C^\alpha(\overline\Omega)$ is positive, 
for some $\alpha\in(0,1)$. 
Let $u$ be a convex solution to \eqref{MAE1}.
Then we have the estimate
	\begin{equation}\label{main}
		\|u\|_{C^{2,\alpha}(\overline\Omega)} \leq C
	\end{equation}
for the same $\alpha$, where $C>0$ is a constant depending on $f$ and the domains. 
\end{theorem}

\vspace{5pt}
The second one is the global $W^{2,p}$ regularity.

\begin{theorem}\label{main01}
Assume that $\Omega, \Omega^*$ are two bounded convex domains in $\mathbb{R}^2$.
Assume that $\partial\Omega, \partial\Omega^* \in C^{1,\varepsilon}$ for some $\varepsilon>0$, and $f\in C^0(\bom)$ is positive.
Let $u$ be a convex solution to \eqref{MAE1}.
Then we have the estimate
\beq\label{esti01}
\|u\|_{W^{2,p}(\bom)}\le C
\eeq
for all $p\ge 1$,
where $C$ is a constant depending only $n, p, f, \Om$, and $\Om^*$.
\end{theorem}

When $f\equiv1$, very recently Savin and Yu \cite{SY} obtained the global $W^{2,p}$ estimate for 
arbitrary bounded convex domains $\Om, \Om^*\subset \R^2$.
Note that the interior $W^{2,p}$ estimate for the Monge-Amp\`ere equation was proved by Caffarelli \cite {C90}
and the global $W^{2,p}$ estimate for the Dirichlet problem was proved by Savin \cite {Sa}.

\vspace{5pt}
Our third regularity result is the global $W^{2,1+\epsilon}$ regularity.

\begin{theorem}\label{epsilon}
Assume that $\Omega, \Omega^*$ are two bounded convex domains in $\mathbb{R}^2$.
Assume that $\lambda^{-1}<f<\lambda$, for a constant $\lambda>0$.  
Let $u$ be a convex solution to \eqref{MAE1}.
Then we have the estimate
\beq\label{epsest}
\|u\|_{W^{2,1+\epsilon}(\bom)}\le C,
\eeq
where $\epsilon, C>0$ are universal constants depending only $\lambda$ and the domains $\Om, \Om^*$.
\end{theorem}

The interior $W^{2,1+\epsilon}$ regularity was obtained by De Philippis, Figalli and Savin in \cite{DFS},
which extends the $W^{2, 1}$ estimate of De Philippis and Figalli \cite {DF}.
Our proof is based on the technique in \cite {DFS} and the uniform density in \cite{C96, CLW}.
By the technique in \cite {DFS} and the uniform density in \cite{CLW}, 
one can also obtain the global $W^{2,1+\epsilon}$ estimate in high dimensions. That is

\begin{thma}
Assume that $\Omega, \Omega^*$ are two bounded convex domains in $\mathbb{R}^n$, $n>2$.
Assume that $\partial\Omega, \partial\Omega^* \in C^{1,1}$, and $\lambda^{-1}<f<\lambda$, for a constant $\lambda>0$.  
Let $u$ be a convex solution to \eqref{MAE1}.
Then we have the estimate
\beq\label{epsestnd}
\|u\|_{W^{2,1+\epsilon}(\bom)}\le C,
\eeq
where $\epsilon, C>0$ are universal constants depending only $n, \lambda$ and the domains $\Om, \Om^*$.
\end{thma}

\vspace{5pt}

This paper is organised as follows:
In \S2, we give some preliminary results on the uniform density and uniform obliqueness properties. 
In \S3, we give the proof of Theorems \ref{thm1} and \ref{main01}, in which the key ingredient is to improve the geometric shape of boundary sub-level sets by an iteration argument. 
In \S4, we prove Theorems \ref{epsilon} and 1.3\textprime.

\vskip5pt

\section{Preliminary}

\subsection{Uniform density}

Let $u$ be a convex solution to \eqref{MAE1}. 
We extend $u$ to a global convex function as follows, 
	\begin{equation*}
		\tilde u(x) := \sup\{\ell(x) : \ell \mbox{ is affine, } \ell \leq u \mbox{ in }\Omega, \nabla\ell\in\Omega^*\}, \quad\mbox{for }x\in\mathbb{R}^2.
	\end{equation*}
For simplicity, we still write the extended function $\tilde u$ as $u$. 
Given a point $x_0\in\overline\Omega$, we denote the ``centred" sub-level set of $u$ at $x_0$ of height $h$ by
	\begin{equation*}
		S_h^c[u](x_0) = \left\{x\in\mathbb{R}^2 : u(x) < \hat\ell(x) + h \right\},
	\end{equation*}
where $\hat\ell$ is an affine function such that $\hat\ell(x_0)=u(x_0)$, and $x_0$ is the mass centre of $S_h^c[u](x_0)$. 	
We also write $S_h^c[u](x_0)$ as $S_h^c[u]$ or $S_h^c(x_0)$ when no confusion arises. 

\begin{lemma}\label{LeUni}
Assume that $\Omega, \Omega^*\subset\R^2$ are bounded convex domains and $0\in\partial\Omega$.
Then there is a positive constant $\delta_0$, independent of $u$ and $h$, such that
	\begin{equation}\label{unid}
		\frac{\Vol\left(\Omega\cap S_h^c(0)\right)}{\Vol\left(S_h^c(0)\right)} \geq \delta_0.
	\end{equation}
\end{lemma}	

The uniform density \eqref{unid} was obtained by Caffarelli in \cite{C96} assuming that $\partial\Omega$ is polynomially convex at $0$. 
As remarked in \cite{C96} that in dimension two, any bounded convex domain is polynomially convex, therefore the uniform density \eqref{unid} is a free gift from \cite{C96}. 
Alternatively, in \cite{CLW} we provide a different proof in any dimension, relaxing the polynomial convexity to the convexity of domains $\Omega, \Omega^*$ with $C^{1,\gamma}$ boundaries for some $\gamma\leq1$.

As a corollary of Lemma \ref{LeUni}, we have the following duality property that
\begin{corollary}\label{dcoro}
Let $v$ be the dual potential function on $\Om^*$, then 
\begin{itemize}
\item[$(i)$] $S^c_h[u]$ is conjugate to $S^c_h[v]$, namely if $A(S^c_h[u])\approx B_{\sqrt{h}}$ by an affine transform $A$,
then $A'^{-1}(S_h^c[v])\approx B_{\sqrt{h}},$ where $A'$ is the transpose of $A$. 

\item[$(ii)$] $\forall$ $x\in S^c_h[u],$ $y\in S^c_h[v]$, we have 
$$|x\cdot y| \lesssim h.$$
\end{itemize}
\end{corollary}

\subsection{Obliqueness}

\begin{lemma}\label{lem obli}
Assume that $\Omega, \Omega^*\subset\R^2$ are two convex domains {with $C^{1,\alpha}$ boundaries for some $\alpha>0$,}
and $f$ is positive and continuous.
Let $0\in\partial\Omega$ and the image $Du(0)=0\in\partial\Omega^*$. 
Then there exists a positive constant $\mu$ such that 
	\begin{equation}\label{sobli}
		\langle \nu(0), \nu^*(Du(0)) \rangle \geq\mu>0,
	\end{equation}
where $\nu$ and $\nu^*$ are the  unit inner normal of $\Om$ and $\Om^*$, respectively.
\end{lemma}

The uniform obliqueness  \eqref{sobli} in dimension two was proved in \cite[\S4.1]{CLW}.
More generally, in \cite[\S4.1]{CLW} it merely assumes that locally $\pom=\{x_2=\rho(x_1)\}$ for a convex function $\rho\geq0$, and there exists an even, continuous function $\sigma$ satisfying $\sigma(t)=o(|t|)$ as $t\to0$ and monotone increasing for $t>0$, such that 
\beqs
\rho(t) \leq \sigma(t)\qquad\mbox{for } t\leq0.
\eeqs
Typically by choosing $\rho(t)=|t|^{1+\alpha}$ for some $\alpha>0$, we then obtain \eqref{sobli}.

\begin{remark}
The uniform obliqueness was proved by Caffarelli \cite{C96} for uniformly convex domains with $C^2$ boundary, 
and by Delano\"e \cite{De} and Urbas \cite {Ur} for uniformly convex domains with $C^{3,1}$ boundary. 
\end{remark}

\vspace{5pt}

\section{Proof of Theorems \ref{thm1} and \ref{main01}}

By the uniform obliqueness \eqref{sobli}, we may assume that locally near the origin, 
$$\Omega=\{x_2\geq \rho(x_1)\},\qquad \Omega^*=\{x_2\geq \rho^*(x_1)\},$$ 
where $\rho, \rho^*\in C^{1,\alpha}$, for some $\alpha>0$, are convex functions satisfying $\rho(0)=\rho^*(0)=0$ and $\rho'(0)=(\rho^*)'(0)=0$.

By John's lemma, we write 
$$S^c_h[u]\approx E_1:=\{x\in\R^2 : \frac{(x_1-k_1x_2)^2}{a_1^2}+\frac{x_2^2}{b_1^2}=1\},$$
and
$$S^c_h[v] \approx E_2:=\{y\in\R^2 : \frac{(y_1-k_2y_2)^2}{a_2^2}+\frac{y_2^2}{b_2^2}=1\}.$$

By global $C^{1,\delta}$ regularity of $u, v$, we have
	\begin{equation}
		b_1, b_2\gtrsim h^{\frac{1}{1+\delta}},\quad\mbox{ for some } \delta\in(0,1).
	\end{equation} 
Denote $\beta:=\frac{1}{1+\delta},$ we have $\beta\in (\frac{1}{2}, 1).$ 
Letting $p$ (resp. $q$) be the intersection of the positive $e_2$-axis and $\p E_1$ (resp. $\p E_2$),
we actually have $p_2, q_2\gtrsim h^\beta.$ 

Note that by the uniform density property, we have 
	\[ |S^c_h[u]|\approx |S^c_h[v]| \approx h, \]
 where $|E|=\Vol(E)$ denotes the Lebesgue measure of the set $E$.

\begin{proof}[Proof of Theorems]
We shall prove Theorems \ref{thm1} and \ref{main01} simultaneously. 
The proof is divided into three steps.

\noindent{\bf Step 1:} 
We show that $u$ is $C^{1,\gamma}$ tangentially for some $\gamma$ that can be computed explicitly.
Recall that we already knew $u\in C^{1,\delta}$ for some $\delta\in(0,1)$, thus tangentially 
	\[ a_1 \gtrsim h^{\frac{1}{1+\delta}} = h^\beta. \]
We shall improve $\delta$ to a bigger $\gamma>\delta$.

Indeed, suppose to the contrary that $a_1 \lesssim h^{\frac{1}{1+\gamma}}$, namely 
	\begin{equation}
		a_1\lesssim h^{\frac{1}{2}+\epsilon},\quad\mbox{ for  $\epsilon$ satisfying }\gamma=\frac{1-2\epsilon}{1+2\epsilon}.
	\end{equation}
Then, since $|S^c_h[u]|\approx h\approx a_1b_1,$ we have 
	\begin{equation}
		b_1\gtrsim h^{\frac{1}{2}-\epsilon.}
	\end{equation}

Take a point $x\in\partial\Omega$ such that $x_1=h^{\frac{1}{2}+\frac{1}{1+\eta}\epsilon}$, where $\eta>0$ is a small constant to be determined later.
 Note that since $\rho\in C^{1,\alpha}$, we have
$$x_2=\rho(x_1)\lesssim h^{(\frac{1}{2}+\frac{1}{1+\eta}\epsilon)(1+\alpha)}.$$

Denote $\tilde{x}=Tx,$ where $T$ is the affine transformation normalising $S^c_h[u],$ namely
\[  \tilde{x}_1=\frac{x_1-k_1x_2}{a_1},\qquad \tilde{x}_2=\frac{x_2}{b_1}\]
such that $TE_1=B_1.$
We \emph{claim} that $x_1-k_1x_2\approx x_1$ (that implies $\tilde{x}_1\rightarrow \infty$ as $h\rightarrow 0$).
Indeed, $k_1\leq \frac{a_1}{p_2}\lesssim h^{\frac{1}{2}+\epsilon-\beta}.$  Hence, $$k_1x_2\lesssim h^{\frac{1}{2}+\epsilon-\beta+(\frac{1}{2}+\frac{1}{1+\eta}\epsilon)(1+\alpha)}.$$
In order for $k_1x_2$ to be a lower order term comparing to $x_1,$ we only need
$$\frac{1}{2}+\epsilon-\beta+(\frac{1}{2}+\frac{1}{1+\eta}\epsilon)(1+\alpha)>\frac{1}{2}+\frac{1}{1+\eta}\epsilon,$$
or equivalently 
\[ \frac{1+\eta+\alpha}{1+\eta}\epsilon > \beta - \frac12(1+\alpha).  \]
It suffices to have $\frac{1}{1+\eta}\epsilon-\beta+\frac{1}{2}(1+\alpha)\geq 0.$ Therefore, we can take 
$$\epsilon=\epsilon_\beta=(\eta+1)(\beta-\frac{1}{2}-\frac{\alpha}{2}).$$
Note that if $\beta\leq \frac{1}{2}+\frac{\alpha}{2},$ we can take $\epsilon$ as any positive small constant.
On the other hand it is straightforward to compute that 
$$\tilde{x}_2=\frac{x_2}{b_1}\lesssim h^{(\frac{1}{2}+\frac{1}{1+\eta}\epsilon)(1+\alpha)-\frac{1}{2}+\epsilon}\rightarrow 0$$
as $h\rightarrow 0.$
Hence, the above estimates and convexity imply that as $h\to0$, the limit domain is independent of $e_1$ direction. Therefore, as long as $\epsilon_\beta>0,$ we obtain that $u$ is
tangentially $C^{1,\gamma}$ with $\gamma=\frac{1-2\epsilon_\beta}{1+2\epsilon_\beta}$, by using the argument in \cite{C96, CLW}.

It means that initially given $u\in C^{1,\delta}$, or equivalently, given a $\beta = \frac{1}{1+\delta}$, we can obtain that $u\in C^{1,\gamma}$ tangentially with $\gamma=\frac{1-2\epsilon_\beta}{1+2\epsilon_\beta}$, or equivalently, improve $\beta$ to a better (smaller) $\beta_1$ with
\[ \beta_1 = \frac{1}{1+\gamma} = \frac12+\epsilon_\beta. \]

\noindent{\bf Step 2.} From the previous step, $h^{\frac{1}{2}+\epsilon_\beta}e_1\in S^c_h[u].$
By Corollary \ref{dcoro}, $S^c_h[v]$ is conjugate to $S^c_h[u]$ and  
$$\left|y\cdot (h^{\frac{1}{2}+\epsilon_\beta}e_1) \right|\lesssim h,\qquad \forall\, y\in S^c_h[v].$$
 It implies that the length of the projection of  $S^c_h[v]$ to $e_1$-axis is bounded by 
$h^{\frac{1}{2}-\epsilon_\beta}.$ Since $|E_2|\approx h,$ we have that $q_2\geq h^{\beta_1},$ where $\beta_1=\frac{1}{2}+\epsilon_\beta.$ 

Recall that originally we merely
have $q_2\gtrsim h^\beta.$ 
Now we \emph{claim} that $\beta-\beta_1\geq c_\alpha$ for some constant $c_\alpha$ depending only on $\alpha.$
Indeed $$\beta-\beta_1=\beta-\frac{1}{2}-(\eta+1)(\beta-\frac{1}{2}-\frac{\alpha}{2})=\frac{\alpha}{2}+\frac{\eta}{2}+\frac{\alpha}{2}\eta-\eta\beta>\frac{\alpha}{4}$$
provided $\eta$ is sufficiently small.

\noindent{\bf Step 3.}
Now, we can repeat {\bf Step 1} for $v$ by replacing $\beta$ by $\beta_1,$ namely using $q_2\geq h^{\beta_1}$, 
Then by {\bf Step 2}, $\beta_1$ is improved to $\beta_2,$ namely in turn, we obtain $p_2\geq h^{\beta_2}$ and
$\beta_1-\beta_2>\frac{\alpha}{4}.$ 
After finite steps, we have $\epsilon_{\beta_m}<0$ for some finite number $m\in\ZZ$. 
This implies that $u, v$ are tangentially $C^{1,\gamma}$ for any $\gamma\in (0,1).$

Once having the tangential $C^{1,\gamma}$ regularity for any $\gamma\in (0,1)$ and the uniform obliqueness \eqref{sobli}, we can use the same argument in \cite[\S5]{CLW} to obtain the global $C^{2,\alpha}$ and $W^{2,p}$ regularities, respectively.
\end{proof}

\vspace{5pt}

\section{Proof of Theorems \ref{epsilon} and 1.3\textprime}
In this section, we prove the global $W^{2,1+\epsilon}$ regularities, by using the geometric properties of boundary sub-level sets established in \cite{CLW}, and the argument of De Philippis-Figalli \cite{DF}, De Philippis-Figalli-Savin \cite{DFS} for the interior $W^{2,1}$, $W^{2,1+\epsilon}$ estimates for Monge-Amp\`ere equations.

For any $x\in \overline{\Omega}\subset\R^n,$ denote the sub-level set 
$$S_h(x)=\{z\in \overline{\Omega} : \ u(z)<u(x)+Du(x)\cdot (z-x)+h\}.$$
Recall the following properties established in \cite[\S2]{CLW}:

\begin{itemize}
\item[$i$)] $S_{h/C}^c(x)\cap \overline{\Omega}\subset S_h(x)\subset S_{Ch}^c(x)\cap \overline{\Omega}$ for a universal constant $C>0$; 

\item[$ii$)] $Du(S^c_h(x))$ is conjugate to $S^c_h(x),$ namely, let $A$ be an affine transformation such that
$A\left(S^c_h(x)\right)\approx B_{h^{\frac{1}{2}}}$, then $A'^{-1} \left(Du(S^c_h(x))\right)\approx B_{h^{\frac{1}{2}}}$, where $A'$ is the transpose of $A$;

\item[$iii$)] $|S_h(x)|\approx |S_h^c(x)|\approx h^{\frac{n}{2}}.$
\end{itemize}

\begin{remark}
Note that the properties $i$)-$iii$) hold in $\R^2$ for any convex domain, see Corollary \ref{dcoro}.
In $\R^n$, $n>2$, we assume the domains are $C^{1,1}$ and convex.
\end{remark}

Caffarelli and Gutierrez \cite{CG} showed that the sub-level sets $S_h(x)\Subset \Omega$ have engulfing properties similar to those of balls. 
In the following lemma, we show that those engulfing properties also hold for the sub-level sets
not necessarily compactly contained in $\Omega$.

\begin{lemma}[Engulfing property]\label{eng}
There exist universal constants $r_0>0$ and $C>1$ such that
for any $x_1\in \overline{\Omega}$, $x_2\in S_h(x_1)$ and $h\leq r_0$, we have $S_h(x_1)\subset S_{Ch}(x_2).$
\end{lemma}

\begin{proof}
Without loss of generality we may assume $x_1=0, u(0)=0$ and $Du(0)=0$. By properties $i$)--$iii$), up to an affine transform $A$, we have
$S_h(0)\approx B_{h^{\frac{1}{2}}}$, $Du(S_h(0))\approx B_{h^{\frac{1}{2}}}.$ 
Hence $|Du(x_2)|\lesssim h^{\frac{1}{2}}.$
Therefore, for any $z\in S_h(0)$ we have 
\begin{equation}
u(z)-\big(u(x_2)+Du(x_2)\cdot (z-x_2)\big)\leq h+|Du(x_2)||z-x_2|\leq Ch,
\end{equation}
which implies that $z\in S_{Ch}(x_2).$
\end{proof}

Using the above lemma it is easy to deduce that 
\begin{lemma}\label{engulfing} There is a universal constant $\delta>0$ such that for any $x_1, x_2\in\overline{\Omega}$ we have 
that if $h_1\leq h_2$ and $S_{\delta h_1}(x_1)\cap S_{\delta h_2}(x_2)\neq \emptyset$, then $S_{\delta h_1}(x_1)\subset S_{h_2}(x_2)$.
\end{lemma}
\begin{proof}
Let $z\in S_{\delta h_1}(x_1)\cap S_{\delta h_2}(x_2)\neq \emptyset,$ by Lemma \ref{eng} we can find some universal constant $C$ so that 
$S_{\delta h_i}(x_i)\subset S_{C\delta h_2}(z)$ for $i=1, 2.$ Then, since $x_2\in S_{C\delta h_2}(z),$ by Lemma \ref{eng} again, we have that 
 $S_{C\delta h_2}(z)\subset  S_{C^2\delta h_2}(x_2).$ The conclusion of the lemma follows by taking $\delta=\frac{1}{C^2}.$
\end{proof}

It is well known that the property of sub-level sets 
stated in Lemma \ref{engulfing} implies the following Vitali covering lemma:
\begin{lemma}[Vitali covering]\label{vital}
 Let $D$ be a compact subset of $\overline{\Omega}$
and let $\{S_{h_x}(x)\}_{x\in D}$ be a family of sub-level sets with $h_x\leq r_0.$ Then, there exist a finite number of 
sub-level sets $\{S_{h_{x_i}}(x_i)\}_{i=1,\ldots, m}$ such that 
$$D\subset \bigcup_{i=1}^mS_{h_{x_i}}(x_i)$$
with $\{S_{\delta h_{x_i}}(x_i)\}_{i=1,\ldots, m}$ disjoint, where $\delta>0$ is a universal constant.
\end{lemma}

For any $x_0\in \overline{\Omega},$  as in \cite{DFS} we define ${\bf a}(S_h(x_0)):=\|A\|^2$ as the normalised size of $S_h(x_0),$ where $A$ is the affine transformation normalising $S_h(x_0)$ such that
$A\left(S_h(x_0)\right)\approx B_{h^{\frac{1}{2}}}.$

\begin{lemma}\label{decay} 
Assume that $\lambda^{-1}<f<\lambda$, for a constant $\lambda>0$.
There exists a universal constant $K>0$ such that
for any $x_0\in \overline{\Omega},$ we have
\begin{equation}\label{decay1}
\int_{S_h(x_0)}\|D^2u\|\leq K{\bf a} \big|\{K^{-1}{\bf a}\leq \|D^2u\|\leq K{\bf a}\}\cap S_{\delta h}(x_0)\big|,
\end{equation}
where ${\bf a}:={\bf a}(S_h(x_0)).$
\end{lemma}

\begin{proof}
By subtracting a linear function we may assume that $u(x_0)=0$, $Du(x_0)=0.$
Let $A$ be the affine transformation such that $A\left(S_h(x_0)\right)\approx B_{h^{\frac{1}{2}}}$ with $\det\, A=1.$ 
Let $\bar{u}(x):=u(A^{-1}x).$ Then, $\det D^2\bar{u}=\tilde{f},$ where $\lambda^{-1}<\tilde{f}(x):=f(A^{-1}x)<\lambda.$
Since $D\bar{u}(A(S_h(x_0)))\approx B_{h^{\frac{1}{2}}},$ we have $|D\bar{u}(z)|\lesssim h^{\frac{1}{2}}$ for any $z\in  A(S_h(x_0))$ and
$|\partial (AS_h(x_0))|\approx h^{\frac{n-1}{2}}.$
Hence, 
\begin{equation}\label{es11}
\int_{S_h(x_0)}\|D^2u\|\leq {\bf a}\int_{A(S_h(x_0))}\|D^2\bar{u}\|={\bf a}\int_{\partial (A(S_h(x_0)))}\bar{u}_\nu\approx {\bf a} h^{\frac{n}{2}}\leq C_1{\bf a}|S_{\delta h}(x_0)|,
\end{equation}
where $C_1$ is a universal constant.

From \eqref{es11} we can deduce that 
\begin{equation}\label{es12}
\big|\{\|D^2\bar{u}\|\leq 2C_1\}\cap A(S_{\delta h}(x_0))\big|\geq \frac{1}{2}|S_{\delta h}|.
\end{equation}
On the other hand, since $\lambda^{-1}<\det D^2\bar{u}<\lambda,$ we have 
\begin{equation}\label{es13}
K^{-1}I\leq D^2\bar{u}<KI\qquad \text{ in}\quad \{\|D^2\bar{u}\|\leq 2C_1\}.
\end{equation}
By the definition of $\bar{u}$  it is straightforward to check that
\begin{equation}\label{es14}
\{K^{-1}I\leq D^2\bar{u}<KI\}\subset A\{K^{-1}{\bf a}\leq \|D^2u\|\leq K{\bf a}\}.
\end{equation}
Now, \eqref{decay1} follows from \eqref{es11}--\eqref{es14}.
\end{proof}

Let $D_k:=\{x\in \overline{\Omega}: \|D^2u(x)\|\geq N^k\},$ where $N$ is some large constant to be determined later. In the following lemma we shall establish a geometric 
decay of $\int_{D_k} \|D^2u\|.$

\begin{lemma}\label{gdecay} There exists a universal constant $N$ such that 
\begin{equation}\label{geom}
\int_{D_{k+1}} \|D^2u\|\leq (1-\tau)\int_{D_k} \|D^2u\|,
\end{equation}
for some universal constant $\tau\in (0,1).$
\end{lemma}

\begin{proof} Let $N\gg K,$ and to be fixed later.
For any $x\in D_{k+1},$ since ${\bf a}(S_h(x))\rightarrow \|D^2u(x)\|$ as $h\rightarrow 0$ and ${\bf a}(S_\delta(x))$ is bounded by some universal constant, we can find a height $h_x$ such that
${\bf a}(S_{h_x}(x))=KN^k.$

Let $\{S_{h_i}(x_i)\}_{i=1}^m$ be a Vitali cover of $D_{k+1}.$ Then, by Lemma \ref{decay} we have 
\begin{equation}\label{decay2}
\int_{S_{h_i}(x_i)}\|D^2u\|\leq K^2N^k \big|\{N^k\leq \|D^2u\|\leq K^2N^k\}\cap S_{\delta h_i}(x_i)\big|,
\end{equation}

Summing up these inequalities and using the facts that $D_{k+1}\subset \bigcup S_{h_i}(x_i)$ and $S_{\delta h_i}(x_i)$ are disjoint, we obtain
\begin{eqnarray}\label{decay3}
\int_{D_{k+1}}\|D^2u\|&\leq& K^2N^k \big|\{N^k\leq \|D^2u\|\leq K^2N^k\}\cap \overline{\Omega}\big|\\
&\leq&C\int_{D_k\backslash D_{k+1}}\|D^2u\|
\end{eqnarray} 
provided $N>K^2.$
Estimate \eqref{geom} follows by adding $C\int_{D_{k+1}}\|D^2u\|$ on both sides of \eqref{decay3} and taking $\tau=1/(1+C).$
\end{proof}

Theorems \ref{epsilon} and {1.3\textprime}  easily follows from Lemma \ref{gdecay} by an elementary computation.

\end{document}